\documentclass[10pt,a4paper]{amsart}
\usepackage[utf8]{inputenc}
\usepackage{amsmath,amsfonts,amssymb}

\usepackage{enumerate}
\usepackage{mathrsfs}
\usepackage{graphicx}
\usepackage{figsize}

\usepackage[]{algorithm2e}
\usepackage{algorithmicx}
\usepackage{pgf,tikz}

\usepackage{float} 

  \newtheorem{theorem}{Theorem}[section]
  \newtheorem{corollary}[theorem]{Corollary}
   
  \newtheorem{lemma}[theorem]{Lemma}

  \newtheorem{remark}[theorem]{Remark}
  \newtheorem{conjeture}{Conjeture}
  \numberwithin{equation}{section}

\def\N{\mathbb{N}}

\def\C{\mathbb{C}}
\def\Z{\mathbb{Z}}

\def\ord{\mathrm{ord}}

\def\G{\mathcal{G}}  
\def\C{\mathcal{C}}

\title{The $\text{t}$-graphs over finitely-generated  groups and the Minkowski metric}

\author{G. Diaz-Porto}
\address{Department of Mathematics and Statistics, Universidad del Norte, Km 5 via a Puerto Colombia, Barranquilla - Colombia}
\email{gabrieladiaz@uninorte.edu.co }
\thanks{ }

\author{I. S. Guti\'errez}
\address{Department of Mathematics and Statistics, Universidad del Norte, Km 5 via a Puerto Colombia, Barranquilla - Colombia}
\email{isgutier@uninorte.edu.co}

\author{A. Torres-Grandisson}
\address{Department of Mathematics and Statistics, Universidad del Norte, Km 5 via a Puerto Colombia, Barranquilla - Colombia}
\email{agrandisson@uninorte.edu.co }
\thanks{ }


\date{ }


\begin{document}
	
\maketitle

\begin{abstract}
In this paper, we introduce the $t$-graphs defined on finitely-generate groups. We study some general aspects of the $t$-graphs on 2-generator groups, emphasising establishing necessary conditions for their connectedness.  In particular, we investigate properties of $t$-graphs defined on finite dihedral groups.
\end{abstract}

\tableofcontents

One of the best-known connections between groups and graph theory was presented by A. Cayley, \cite{Cayley}. He gave a group $G$ as a directed graph, where the vertices correspond to elements of $G$ and the edges to multiplication by group generators and their inverses. Such a graph is called a Cayley diagram or Cayley graph of $G$. It is a central tool in combinatorial and geometric group theory. 

Recent works reveal many different ways of associating a graph to a given finite group, most of them inspired by a question posed by P. Erdös \cite{Neumann}. These differences lie in the adjacency criterion used to relate two group elements constituting the set of vertices of such a graph. Some essential authors in this context are A. Abdollahi \cite{Abdollahi}, A. Lucchini \cite{Lucchini 1, Lucchini 2}, and D. Hai-Reuven \cite{Reuven}, among others.

Our notation will be standard, and it is as in \cite{Huppert}, and \cite{Melissa} for groups and graphs. Let $G = \langle g_1,\ldots, g_n\rangle$ be a finitely-generated group, and suppose now that every element in $g\in G$ can be uniquely written as 
\begin{equation}\label{escritura}
	g=\prod_{i=1}^n g_i^{\epsilon_i}.
\end{equation}
To determine a measure of the separation between two elements of $G$ we introduce the following distance map $d_1 : G\times G \longrightarrow \N_0$, defined by 
\begin{equation}\label{d1}
	d_1(g,h)= d_1\Big(\prod_{i=1}^n g_i^{\epsilon_i}, \prod_{i=1}^n g_i^{\delta_i}\Big) = \sum_{i=1}^n |\epsilon_i-\delta_i|.
\end{equation}
The set $G$ endowed with this distance $d$ is a metric space. Note that $d_1$ is just the Minkowski $l_p$ metric for $p=1$. This is also called taxicab distance, Manhattan distance, or grid distance.

G. Diaz  introduced in \cite{Diaz} the $t$-graphs using Minkowski's metric. This graph can be defined by having the group $G$ as the underlying set of vertices and the following adjacency criteria: Let $t$ be an integer number with $1\leq t \leq n$. We say that $g, h\in G$ are adjacent if and only if  $d_1(g,h)= t$.

The simplest example is when $G$ is a cyclic group. 	Let $G=\langle g\rangle$ be a cyclic group with finite order $m$. That is, $G=\{ 1, g, \cdots, g^{m-1} \}$. From \eqref{d1} we have that 
\begin{equation}\label{d-cyclic}
	d_1(g^i, g^j)=|i-j|, \ \ \text{for all $0\leq i, j \leq m-1$}.
\end{equation}
It means that in the $t$-graph of $G$  there exist an edge between $g^i$ and $g^j$  if and only if $|i-j|=t$. Defining  over $G$ the following relation:
\begin{equation}\label{cyclic2}
	g^i\sim g^j \iff i\equiv j \bmod t,
\end{equation}
we have that $\sim$ is an equivalence relation and then we have a partition of $G$ in $t$ classes given by
\begin{equation}\label{cyclic3}
	[g^i]:=\{ g^j\in G\mid j\equiv i \bmod t\},
\end{equation}
where $i\in \{0, 1, \ldots, t-1\}$. Then the $t$-graph of a cyclic group $G$ can be viewed as the union of $t$ connected components, consisting of paths graphs  or isolated points. Consequently for $t\geq 2$ the $t$-graph is non connected and 2-chromatic. The $1$-graph of $G$ is a path graph, and then connected. 

If $t$ is a divisor of the group order $m$, then it is well known that $G$ has a cyclic subgroup $U$ of order $m/t$, and the elements of $U$ form a subgraph with $m/n$ vertices, which is a connected component of the $t$-graph of $G$.

An immediate consequence of this is that if $G$ is a finite abelian group, say  $G= \langle g_1\rangle \times \cdots \times \langle g_n\rangle$, with $\ord(g_j) = \epsilon_j$. Then the 1-graph of $G$ is the  Cartesian product of $n$ paths graphs of lengths $\epsilon_j$ respectively. That is, a  $n$-dimensional square grid graph. In general, using the above example  the $t$-graph of $G$ is the Cartesian product of $t$ components.  

On the other hand, It is well known that the dihedral group $D_n$ and the quaternion group $Q_8$ have the following  group presentation respectively:
\begin{align}\label{diedrico}
	D_n &= \langle a, b \mid a^2 = b^n =1, \ a b a = b^{-1}\rangle,\\
	Q_8 &= \langle a, b \mid a^4 =1, \ a^2 = b^2, \ b a b^{-1}  = a^{-1}\rangle.
\end{align}
Furthermore 
\begin{equation}\label{abelian}
	\Z_2\times \Z_4 = \langle a, b \mid a^2 = b^4= 1, \ ab = ba\rangle.
\end{equation} 
Note that, in terms of their generators, the elements of $\Z_2\times \Z_4$, $D_4$ and $Q_8$ can be written as follows
\begin{equation}
	\{ 1, a, b, b^2, b^3, a b, a b^2, a b^3 \},
\end{equation}
and we have the following distance table:
\begin{table}[H]
	\centering
	\begin{tabular}{|c|c|c|c|c|c|c|c|c|}
		\hline
		$d_1$ & 1 & $a$ & $b$ & $b^2$ & $b^3$ & $a b$ & $a b^2$ & $a b^3$\\
		\hline
		1 & 0 & 1 & 1 & 2 & 3 & 2 & 3 & 4 \\
		$a$ & 1 & 0 & 2 & 3 & 4 & 1 & 2 & 3 \\
		$b$ & 1 & 2 & 0 & 1 & 2 & 1 & 2 & 3 \\
		$b^2$ & 2 & 3 & 1 & 0 & 1 &	2 &	1 & 2 \\
		$b^3$ & 3 &	4 &	2 &	1 &	0 &	3 &	2 &	1 \\
		$a b$ & 2 & 1 & 1 & 2 & 3 & 0 & 1 & 2 \\
		$a b^2$ & 3 & 2 & 2 & 1 & 3 & 1 & 0 & 1 \\
		$a b^3$ & 4 & 3 & 3 & 2 & 1 & 2 & 1 & 0\\
		\hline
	\end{tabular}
	\caption{Table of distances of $\Z_2\times \Z_4$, $D_4$ and $Q_8$}
	\label{tabla d8}
\end{table}

An illustration of the first four $t$-graph of these three groups is presented in the following figure.

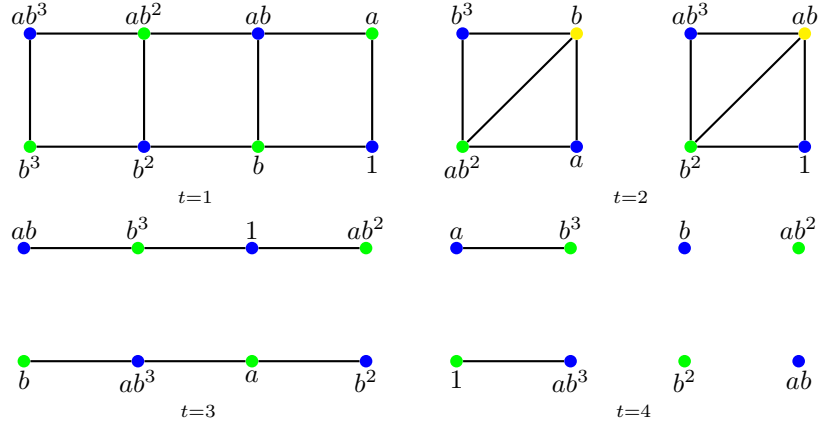
\begin{figure}[H]
	\centering
	$\underset{t=1}{
		\begin{tikzpicture}
			\node[circle,scale=0.5,fill=blue] (1) at (1,0)  {};		
			\node[circle,scale=0.5,fill=green] (a) at (1,1.5)  {};				
			\node[circle,scale=0.5,fill=blue] (ab) at (-0.5,1.5)  {};		
			\node[circle,scale=0.5,fill=green] (ab2) at (-2,1.5)  {};					
			\node[circle,scale=0.5,fill=blue] (ab3) at (-3.5,1.5)  {};			
			\node[circle,scale=0.5,fill=green] (b3) at (-3.5,0)  {}; 		
			\node[circle,scale=0.5,fill=blue] (b2) at (-2,0)  {};		
			\node[circle,scale=0.5,fill=green] (b) at (-0.5,0)    {}; 		 		
			\draw[thick] (1) -- (a) -- (ab) -- (ab2)  -- (ab3) -- (b3)   -- (b2) -- (b) -- (1)	;
			\draw[thick]  (b) -- (ab) 	;
			\draw[thick] (b2)  -- (ab2) 	;
			\node [below] at (1,0) {$1$};
			\node[above] at (1,1.5)  {$a$};
			\node[above] at (-0.5,1.5) {$ab$};
			\node[above] at  (-2,1.5) {$ab^2$};
			\node[above] at (-3.5,1.5) {$ab^3$};
			\node[below] at (-3.5,0) {$b^3$};
			\node[below] at (-2,0) {$b^2$};	
			\node[below] at (-0.5,0)  {$b$};
	\end{tikzpicture}}$
	\hskip0.5cm
	$\underset{t=2}{	
		\begin{tikzpicture}
			\node[circle,scale=0.5,fill=blue] (1) at (1,0)  {};		
			\node[circle,scale=0.5,fill=yellow] (a) at (1,1.5)  {};				
			\node[circle,scale=0.5,fill=blue] (ab) at (-0.5,1.5)  {};		
			\node[circle,scale=0.5,fill=yellow] (ab2) at (-2,1.5)  {};					
			\node[circle,scale=0.5,fill=blue] (ab3) at (-3.5,1.5)  {};			
			\node[circle,scale=0.5,fill=green] (b3) at (-3.5,0)  {}; 		
			\node[circle,scale=0.5,fill=blue] (b2) at (-2,0)  {};		
			\node[circle,scale=0.5,fill=green] (b) at (-0.5,0)    {}; 		 		
			\draw[thick] (1) -- (a) -- (ab)  (ab2)  -- (ab3) -- (b3)   -- (b2)  (ab2) -- (b3) (a)--(b) -- (1)	;
			\draw[thick]  (b) -- (ab) 	;
			\draw[thick] (b2)  -- (ab2) 	;
			\node [below] at (1,0) {$1$};
			\node[above] at (1,1.5)  {$ab$};
			\node[above] at (-0.5,1.5) {$ab^3$};
			\node[above] at  (-2,1.5) {$b$};
			\node[above] at (-3.5,1.5) {$b^3$};
			\node[below] at (-3.5,0) {$ab^2$};
			\node[below] at (-2,0) {$a$};	
			\node[below] at (-0.5,0)  {$b^2$};
	\end{tikzpicture}}$
	
	$\underset{t=3}{
		\begin{tikzpicture}
			\node[circle,scale=0.5,fill=blue] (1) at (1,0)  {};		
			\node[circle,scale=0.5,fill=green] (a) at (1,1.5)  {};				
			\node[circle,scale=0.5,fill=blue] (ab) at (-0.5,1.5)  {};		
			\node[circle,scale=0.5,fill=green] (ab2) at (-2,1.5)  {};					
			\node[circle,scale=0.5,fill=blue] (ab3) at (-3.5,1.5)  {};			
			\node[circle,scale=0.5,fill=green] (b3) at (-3.5,0)  {}; 		
			\node[circle,scale=0.5,fill=blue] (b2) at (-2,0)  {};		
			\node[circle,scale=0.5,fill=green] (b) at (-0.5,0)    {}; 		 		
			\draw[thick]  (a) -- (ab) -- (ab2)  -- (ab3)  (b3)   -- (b2) -- (b) -- (1)	;
			\node [below] at (1,0) {$b^2$};
			\node[below] at (-0.5,0)  {$a$};
			\node[below] at (-2,0) {$ab^3$};	
			\node[below] at (-3.5,0) {$b$};
			\node[above] at (1,1.5)  {$ab^2$};
			\node[above] at (-0.5,1.5) {$1$};
			\node[above] at  (-2,1.5) {$b^3$};
			\node[above] at (-3.5,1.5) {$ab$};
	\end{tikzpicture}}$
	\hskip0.5cm
	$\underset{t=4}{
		\begin{tikzpicture}
			\node[circle,scale=0.5,fill=blue] (1) at (1,0)  {};		
			\node[circle,scale=0.5,fill=green] (a) at (1,1.5)  {};				
			\node[circle,scale=0.5,fill=blue] (ab) at (-0.5,1.5)  {};		
			\node[circle,scale=0.5,fill=green] (ab2) at (-2,1.5)  {};					
			\node[circle,scale=0.5,fill=blue] (ab3) at (-3.5,1.5)  {};			
			\node[circle,scale=0.5,fill=green] (b3) at (-3.5,0)  {}; 		
			\node[circle,scale=0.5,fill=blue] (b2) at (-2,0)  {};		
			\node[circle,scale=0.5,fill=green] (b) at (-0.5,0)    {}; 		 		
			\draw[thick]   (a) (ab) (ab2)  -- (ab3)  (b3)   -- (b2)  (b) (1)	; 
			\node[above] at (-3.5,1.5) {$a$};
			\node[above] at  (-2,1.5) {$b^3$};
			\node[above] at (-0.5,1.5) {$b$};
			\node[above] at (1,1.5)  {$ab^2$};
			\node [below] at (1,0) {$ab$};
			\node[below] at (-0.5,0)  {$b^2$};
			\node[below] at (-2,0) {$ab^3$};	
			\node[below] at (-3.5,0) {$1$};
	\end{tikzpicture}}$
	\caption{The $t$-graphs of $\Z_2\times \Z_4$,  $D_4$ and $Q_8$.}
	\label{D8}
\end{figure}

Despite being non-isomorphic groups, these groups have precisely the same $t$-graphs since the metric used to define the adjacency criterion only considers the writing of the group's elements and not how they interact with each other. This leads to the conclusion that any 2-generator group $G =  \langle a, b \rangle$ in which every element can be written in the form $a^ib^j$, with $0\leq i\leq \ord(a)-1$ and $0\leq j\leq \ord(b)-1$ has the same $t$-graphs as the group $\Z_{\ord(a)}\times \Z_{\ord(b)}$, since when considering the form its elements are written in terms of the generators, the underlying sets are the same.

Therefore, in principle, to study the $t$-graphs of a group $G$, it is sufficient to consider abelian groups, expressed as products of cyclic groups. Naturally, this implies asking oneself, given an arbitrary group $G$, how to determine the abelian group with which it will share the same $t$-graphs. For example, the symmetric group of degree 5 has the same $t$-graph as $\Z_2\times \Z_3\times \Z_4\times \Z_5$.

On the other hand, this situation opens the possibility of studying $t$-graphs  by defining the adjacency criterion in terms of another metric. This change could imply that the group structure plays a more critical role. 

The main goal of this paper is to obtain some characterizations of the $t$-graphs $\G$ associated with 2-generator group $G$ that can be expressed in the form 
\begin{equation}\label{2-gen}
	G =  \langle a, b \rangle = \{a^ib^j\mid 0\leq i\leq m, \ \ 0\leq j\leq n\}.
\end{equation}
where $m\leq \ord(a)$ y $n\leq \ord(b)$; $n, m\in \Z$. These numbers $m$ and $n$ depend exclusively on the structure, namely on the group's presentation and the order of $G$.
We concretely determine the number of connected components of $\G$ depending on whether $t$ is an even or odd number.

\section{Preliminaries on t-graphs}

\begin{lemma}
	Let $G=\langle g_1, \cdots, g_n \rangle$ be a finitely generated group and $H\leq G$. Then the $t$-graph of $H$ is a subgraph of the $t$-graph of  $G$.
\end{lemma}

\begin{proof} 
	It follows immediately from the definition of $t$-graph.
\end{proof}

\begin{lemma}\label{isomorfia}
	Let $G=\langle g_1, \cdots, g_n\rangle$ and $H=\langle h_1, \cdots, h_n\rangle$ be finitely generated groups. If $G$ and $H$ are isomorphic, then  the corresponding $t$-graphs are isomorphic, for all natural number $t$.
\end{lemma}

\begin{proof} 
	Let $f: G \longrightarrow H$ be a group isomorphism with $f(g_i)=h_i$, and let   $\G=(G, E_1)$ and $\mathcal{H}=(H,E_2)$ be the corresponding $t$-graphs of $G$ and  $H$ respectively. Let suppose  $\{x, y\}\in E_1$, with $x= \prod_{i=1}^n g_i^{\epsilon_i}$ and $y= \prod_{i=1}^n g_i^{\delta_i}$. Then $d_1(x,y) = t$, and we have
	\begin{align*}
		d_1(f(x), f(y))  & = d_1\Big(\prod_{i=1}^n f(g_i)^{\epsilon_i}, \prod_{i=1}^n f(g_i)^{\delta_i}\Big)\\
		& = d_1\Big(\prod_{i=1}^n h_i^{\epsilon_i}, \prod_{i=1}^n h_i^{\delta_i})\Big)\\
		&=\displaystyle{\sum_{i=1}^n |\epsilon_i-\delta_i|}\\
		&= d_1(x,y).
	\end{align*}	
	It follows that  $\{f(x), f(y)\}\in E_2$.
\end{proof}

\begin{remark}
	Note that the reciprocal of the statement in lemma \ref{isomorfia} is in general not true. For example, $t$-graphs of the dihedral $D_4$ and the quaternions group, $Q_8$ are isomorphic even though $D_4 \not\cong Q_8$. 
\end{remark}

To study $t$-graphs in the given context, it is necessary to use the spectral theory of graphs, which consists of studying the properties of the Laplacian  matrix of a graph, more specifically its eigenvalues and eigenvectors.  

The Laplacian  matrix of $\G =(V,E)$  is the $n\times  n$ matrix $L = (l_{ij})$ indexed by $V$, whose $(i,j)$-entry is defined as follows
\begin{equation}
	l_{ij} = \begin{cases}
		-1 & \text{if} \ \{v_i, v_j\} \in E\\
		\deg(v_i) & \text{if} \ i=j\\
		0 & \text{otherwise}.
	\end{cases}
\end{equation}

To analyse the behaviour of the number of connected components $k(\G)$ of the t-graphs defined on a group $G$, we use the following Theorem, which allows us to realise the tables  \ref{TF1} and \ref{TF2}. A proof of this Theorem can be found in \cite[Theorem 7.1]{Nica}. 

\begin{theorem}\label{T1}
	A graph $\G$ has $k$ connected components if and only if the algebraic multiplicity of $0$ as a Laplacian eigenvalue is $k$.
\end{theorem}

In the following, to study the $t$-graphs associated with a finite group $G$, we will consider only finite 2-generator groups, which can be expressed in the form \eqref{2-gen}. These numbers $m$ and $n$ depend exclusively on the structure, namely on the group's presentation and the order of $G$.

Let $G$ be  such a group.  To observe the behaviour of the number of connected components $k(\G)$ of a $t$-graph $\G$ determined by a group $G$ we make use of Theorem \ref{T1}, with which we were able to make the following tables: 

\begin{table}[H]
	\begin{center}
		\resizebox{13cm}{!}{	\begin{tabular}{|c|c|c|c|c|c|c|c|c|c|c|c|c|c|c|c|c|c|c|c|c|}
				\hline
				\textbf{$n$\textbackslash $t$} & 1 & 2 & 3 & 4 & 5 & 6  & 7  & 8  & 9  & 10 & 11 & 12 & 13 & 14 & 15 & 16 & 17 & 18 & 19 & 20   \\ 
				\hline
				2                 & 1 & 2 & - & - & - & -  & -  & -  & -  & -  & -  & -  & -  & -  & -  & -  & -  & -  & -  & -   \\ \hline
				3                 & 1 & 2 & 4 & - & - & -  & -  & -  & -  & -  & -  & -  & -  & -  & -  & -  & -  & -  & -  & -     \\ \hline
				4                 & 1 & 2 & 2 & 6 & - & -  & -  & -  & -  & -  & -  & -  & -  & -  & -  & -  & -  & -  & -  & -     \\ \hline
				5                 & 1 & 2 & 1 & 4 & 8 & -  & -  & -  & -  & -  & -  & -  & -  & -  & -  & -  & -  & -  & -  & -       \\ \hline
				6                 & 1 & 2 & 1 & 2 & 6 & 10 & -  & -  & -  & -  & -  & -  & -  & -  & -  & -  & -  & -  & -  & -      \\ \hline
				7                 & 1 & 2 & 1 & 2 & 4 & 8  & 12 & -  & -  & -  & -  & -  & -  & -  & -  & -  & -  & -  & -  & -      \\ \hline
				8                 & 1 & 2 & 1 & 2 & 2 & 6  & 10 & 14 & -  & -  & -  & -  & -  & -  & -  & -  & -  & -  & -  & -      \\ \hline
				9                 & 1 & 2 & 1 & 2 & 1 & 4  & 8  & 12 & 16 & -  & -  & -  & -  & -  & -  & -  & -  & -  & -  & -      \\ \hline
				10                & 1 & 2 & 1 & 2 & 1 & 2  & 6  & 10 & 14 & 18 & -  & -  & -  & -  & -  & -  & -  & -  & -  & -      \\ \hline
				11                & 1 & 2 & 1 & 2 & 1 & 2  & 4  & 8  & 12 & 16 & 20 & -  & -  & -  & -  & -  & -  & -  & -  & -      \\ \hline
				12                & 1 & 2 & 1 & 2 & 1 & 2  & 2  & 6  & 10 & 14 & 18 & 22 & -  & -  & -  & -  & -  & -  & -  & -       \\ \hline
				13                & 1 & 2 & 1 & 2 & 1 & 2  & 1  & 4  & 8  & 12 & 16 & 20 & 24 & -  & -  & -  & -  & -  & -  & -         \\ \hline
				14                & 1 & 2 & 1 & 2 & 1 & 2  & 1  & 2  & 6  & 10 & 14 & 18 & 22 & 26 & -  & -  & -  & -  & -  & -         \\ \hline
				15                & 1 & 2 & 1 & 2 & 1 & 2  & 1  & 2  & 4  & 8  & 12 & 16 & 20 & 24 & 28 & -  & -  & -  & -  & -         \\ \hline
				16                & 1 & 2 & 1 & 2 & 1 & 2  & 1  & 2  & 2  & 6  & 10 & 14 & 18 & 22 & 26 & 30 & -  & -  & -  & -          \\ \hline
				17                & 1 & 2 & 1 & 2 & 1 & 2  & 1  & 2  & 1  & 4  & 8  & 12 & 16 & 20 & 24 & 28 & 32 & -  & -  & -       \\ \hline
				18                & 1 & 2 & 1 & 2 & 1 & 2  & 1  & 2  & 1  & 2  & 6  & 10 & 14 & 18 & 22 & 26 & 30 & 34 & -  & -        \\ \hline
				19                & 1 & 2 & 1 & 2 & 1 & 2  & 1  & 2  & 1  & 2  & 4  & 8  & 12 & 16 & 20 & 24 & 28 & 32 & 36 & -          \\ \hline
				20                & 1 & 2 & 1 & 2 & 1 & 2  & 1  & 2  & 1  & 2  & 2  & 6  & 10 & 14 & 18 & 22 & 26 & 30 & 34 & 38       \\ \hline
		\end{tabular}}
	\end{center}
	\caption{Number of connected components of the $t$-graphs on $\Z_n\times \Z_2$.}
	\label{TF1}
\end{table}

\begin{table}[H]
	\begin{center}
		\resizebox{13cm}{!}{		\begin{tabular}{|c|c|c|c|c|c|c|c|c|c|c|c|c|c|c|c|c|c|c|c|c|}
				\hline
				\textbf{$n$\textbackslash $t$} & 1 & 2 & 3 & 4 & 5  & 6  & 7  & 8  & 9  & 10 & 11 & 12 & 13 & 14 & 15  & 16  & 17  & 18  & 19  & 20   \\ 
				\hline
				2                 & 1 & 2 & 4 & - & -  & -  & -  & -  & -  & -  & - & -  & -  & -  & -  & -  & -  & - & -  & -     \\ \hline
				3                 & 1 & 2 & 2 & 7 & -  & -  & -  & -  & -  & -  & - & -  & -  & -  & -  & -  & -  & - & -  & -    \\ \hline
				4                 & 1 & 2 & 1 & 4 & 10 & -  & -  & -  & -  & -  & - & -  & -  & -  & -  & -  & -  & - & -  & -  \\ \hline
				5                 & 1 & 2 & 1 & 3 & 7  & 13 & -  & -  & -  & -  & - & -  & -  & -  & -  & -  & -  & - & -  & -    \\ \hline
				6                 & 1 & 2 & 1 & 2 & 4  & 10 & 16 & -  & -  & -  & - & -  & -  & -  & -  & -  & -  & - & -  & -     \\ \hline
				7                 & 1 & 2 & 1 & 2 & 2  & 7  & 13 & 19 & -  & -  & - & -  & -  & -  & -  & -  & -  & - & -  & -    \\ \hline
				8                 & 1 & 2 & 1 & 2 & 1  & 4  & 10 & 16 & 22 & -  & - & -  & -  & -  & -  & -  & -  & - & -  & -    \\ \hline
				9                 & 1 & 2 & 1 & 2 & 1  & 3  & 7  & 13 & 19 & 25 & - & -  & -  & -  & -  & -  & -  & - & -  & -   \\ \hline
				10                & 1 & 2 & 1 & 2 & 1  & 2  & 4  & 10 & 16 & 22 & 28 & -  & -  & -  & -  & -  & -  & - & -  & -    \\ \hline
				11                & 1 & 2 & 1 & 2 & 1  & 2  & 2  & 7  & 13  & 19  & 25 & 31  & -  & -  & -  & -  & -  & - & -  & -     \\ \hline
				12                & 1 & 2 & 1 & 2 & 1  & 2  & 1  & 4  & 10  & 16  & 22 & 28  & 34  & -  & -  & -  & -  & - & -  & -    \\ \hline
				13                & 1 & 2 & 1 & 2 & 1  & 2  & 1  & 3  & 7  & 13  & 19 & 25  & 31  & 37  & -  & -  & -  & - & -  & -     \\ \hline
				14                & 1 & 2 & 1 & 2 & 1  & 2 & 1  & 2  & 4  & 10  & 16 & 22  & 28  & 34  & 40  & -  & -  & - & -  & -    \\ \hline
				15                & 1 & 2 & 1 & 2 & 1  & 2 & 1 & 2  & 2  & 7  & 13 & 19  & 25  & 31  & 37  & 43  & -  & - & -  & -    \\ \hline
				16                & 1 & 2 & 1 & 2 & 1  & 2  & 1 & 2 & 1  & 4  & 10 & 16  & 22  & 28  & 34  & 40  & 46  & - & -  & -    \\ \hline
				17                & 1 & 2 & 1 & 2 & 1  & 2  & 1 & 2 & 1 & 3  & 7 & 13  & 19  & 25  & 31  & 37  & 43  & 49 & -  & -   \\ \hline
				18                & 1 & 2 & 1 & 2 & 1  & 2  & 1  & 2 & 1 & 2 & 4 & 10  & 16  & 22  & 28  & 34  & 40  & 46 & 52  & -    \\ \hline
				19                & 1 & 2 & 1 & 2 & 1  & 2  & 1  & 2 & 1 & 2 & 2 & 7  & 13  & 19  & 25  & 31  & 37  & 43 & 49  & 55  \\ \hline
				20                & 1 & 2 & 1 & 2 & 1  & 2  & 1  & 2 & 1 & 2 & 1 & 4  & 10  & 16  & 22  & 28  & 34  & 40 & 46  & 52   \\ \hline
		\end{tabular}}
	\end{center}
	\caption{Number of connected components of the $t$-graphs on $\Z_n\times \Z_3$.}
	\label{TF2}
\end{table}

Note in the previous tables  that $k(\G)$ has a regular behaviour up to a certain value of $t$ where if $t$ is even $k(\G) = 2$ and if t is odd then $k(\G) = 1$ and from this value onwards $k(\G)$ has a value with no definite pattern.  When $m = 2$, (for example by dihedral groups) the unstable  behaviour of the matrix entry when $t > \left\lceil \frac{m+n-2}{2} \right\rceil$ shows a pattern in which the number of connected components of the $t$-graph increases by 4.  On the other hand, for $m=3$ the non regular part starts with 7 connected components  and so progresses from 6 to 6 if $n$ is odd, and starts at 4 and progresses from 6 to 6 when $n$ is even.

This fact leads us to state the following theorem, which allows us to have a first characterisation of the $t$-graphs associated with 2-generate groups in the form \eqref{2-gen} concerning the number of connected components.

\begin{lemma}
	Let $G$ be a 2-generator group in the form \eqref{2-gen} with  $n, m \geq 2$, and $\G $ the corresponding $t$-graph of $G$. Then $\mathcal{G}$ has no isolated points if and only if $t\leq \left\lceil \frac{m+n-2}{2} \right\rceil$.
\end{lemma}

\begin{proof}
	Let $x=a^ib^j,  y=a^kb^l \in G$ with 
	\begin{equation}\label{isolated}
		d_1(x,y) = |i-k| + |j-l| =t.
	\end{equation}
	Then $t\in \{0, \ldots, m+n-2\}$, and suppose $|i-k| = s\in \{0, \ldots, m-1\}$. This implies that $|j-l| = t-s\in \{0, \ldots, n-1\}$. Note that if $t-s>n-1$ the  equality \eqref{isolated}  is not verified. That is, there is no an edge between $x$ and $y$.  Then, in order not to have isolated points, it must be fulfilled that $t-s\leq n-1$, with  $s\in \{0, \ldots, m-1\}$. Moreover, $t\leq n-1$. Analogously, it follows that $t\leq m-1$. Consequently, $2t\leq m+n-2$, and therefore $t\leq \left\lceil \frac{m+n-2}{2} \right\rceil$.  
\end{proof}

\begin{theorem}\label{T2}
	Let $G$ be a 2-generator group in the form \eqref{2-gen} with  $n, m \geq 2$, and $\G = (G,E)$ the corresponding $t$-graph, with $t\leq \left\lceil \frac{m+n-2}{2} \right\rceil$. 
	\begin{enumerate}[(1)]
		\item If $t$ is an even number, then  $k(\G)=2$. 
		\item If $t$ is an odd number, then $\G$ is connected.
	\end{enumerate}
\end{theorem}

\begin{proof}
	From the above lemma we have that the condition $t\leq \left\lceil \frac{m+n-2}{2} \right\rceil$  implies that $\G$ has no isolated points. We now differentiate two possible cases.
	\begin{enumerate}[(1)]
		\item Let $t$ be an even number. We define $\C_1 = (V_1, E_1)$ and $\C_2 = (V_2, E_2)$ subgraph of $\G$, as follows:
		\begin{align}
			V_1 & := \{a^ib^j \mid  i+j \equiv 0 \bmod 2 \},\\
			E_1	& := \{\{a^ib^j, a^kb^l\} \mid i+j, k+l \equiv 0 \bmod 2 \wedge \ |i-k|+|j-l|=t \},
		\end{align}
		and 
		\begin{align}
			V_2 & := \{a^ib^j\mid  i+j\equiv 1 \bmod 2 \},\\
			E_2	& := \{\{a^ib^j, a^kb^l\} \mid i+j, k+l \equiv 1 \bmod 2 \wedge \ |i-k|+|j-l|=t \}.
		\end{align}
		It is clear that $V_1\cup V_2 = G$, and then $k(\G)=2$.

		\item Let $t$ be an even number, and $x=a^ib^j \in G$ arbitrary. If $i+j \equiv 1 \bmod 2$, then we consider the sets
		\begin{align}
			& \{a^kb^l \mid i, k+l \equiv 0 \bmod 2,  j\equiv 1\bmod 2 \wedge \ |i-k|+|j-l|=t \}\\ 
			& \{a^kb^l \mid j, k+l \equiv 0 \bmod 2,  i\equiv 1\bmod 2 \wedge \ |i-k|+|j-l|=t \}.
		\end{align}
		Since $\mathcal{G}$ has no isolated points, at least one of these sets is non empty, and then $\{a^ib^j, a^kb^l\} \in E$.
		
		If $i+j \equiv 0 \bmod 2$, then a similar analysis leads to the same conclusion. Then we have that $\G$ is a connected graph.
	\end{enumerate}
\end{proof}

\section{The t-graph of some 2-generator groups}

This section considers the $t$-graph of a particular case of 2-generator groups. Specifically, we suppose that $a$ is an involution and $b$ has order $n$. For example, the group $G$ considered can be the abelian group $\Z_2\times \Z_n$ or the dihedral group $D_n$ of order $n$. 

The first theorem shows that the 1-graph associated with a finite dihedral group $D_n$ has a simple structure. It  corresponds to a square $(n\times 2)$-grid, as shown in the figure \ref{Dn1} below. Therefore, this graph is bi-chromatic or bipartite.

\begin{theorem}\label{E2}
	The $1$-graph of   $D_n$ is bipartite.
\end{theorem}

\begin{proof} 
	From \eqref{diedrico} we have that 
	\begin{equation}\label{diedrico2}
		D_n = \{1,b, \cdots, b^{n-1}\} \cup \{a b, \ldots, a b^{n-1} \}.
	\end{equation}
	Note that 
	\begin{equation}
		d_1(b^i,b^{i+1})= d_1(ab^i,ab^{i+1})=1,
	\end{equation}
	then the sets $\{1,b, \cdots, b^{n-1}\}$ and $\{a, a b, \cdots, ab^{n-1}\}$ form a bipartition of the vertex set  $D_n$.  
\end{proof}

\begin{figure}[H]
	\label{Dn1}
	\begin{center}
		\begin{tikzpicture}[mynode/.style={font=\color{#1}\sffamily,circle,inner sep=2pt,minimum size=0.2cm,draw,thick}] 
			\node[mynode=yellow,fill=blue] (1) at (1,0) {};
			\node[mynode=yellow,fill=green] (2) at (3,0) {};
			\node[mynode=black,fill=blue] (3) at (5,0) {};
			\node[mynode=green,fill=green] (4) at (7,0) {}; 
			\node[mynode=green,fill=blue] (5) at (9,0) {}; 
			\node[mynode=green,fill=green] (6) at (11,0) {};    
			\node[mynode=green,fill=blue] (6a) at (11,2) {}; 
			\node[mynode=green,fill=green] (5a) at (9,2) {}; 
			\node[mynode=green,fill=blue] (4a) at (7,2) {}; 
			\node[mynode=black,fill=green] (3a) at (5,2) {};
			\node[mynode=yellow,fill=blue] (2a) at (3,2) {};
			\node[mynode=yellow,fill=green] (1a) at (1,2) {};
			\draw[thick] (1) --(2) --(3) --(4)  -- (5) -- (6) -- (6a)  -- (5a) -- (4a) -- (3a) -- (2a) -- (1a) -- (1) 	;
			\draw[thick]  (2) -- (2a) (3) -- (3a) (4) -- (4a) (5) -- (5a)	;
			\node [below] at (1,0) {$b^{n-1}$};
			\node [below] at (3,0) {$b^{n-2}$};
			\node [below] at (5,-0.1) {$\cdots$};
			\node [below] at (7,0) {$b^2$};
			\node [below] at (9,0) {$b$}; 
			\node [below] at (11,0) {$1$}; 
			\node [above] at (11,2) {$a$}; 
			\node [above] at (9,2) {$ab$}; 
			\node [above] at (7,2) {$ab^2$};
			\node [above] at (5,2) {$\cdots$};
			\node [above] at (3.1,2) {$ab^{n-2}$};
			\node [above] at (1.1,2) {$ab^{n-1}$}; 
		\end{tikzpicture}
		\caption{The 1-graph of $D_n$.}
	\end{center}
\end{figure}
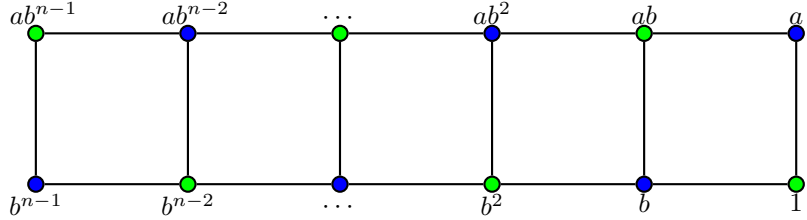

The theorem \ref{T2}  leads to a complete characterisation of the $t$-graphs associated with $D_n$.  However, before we start to characterise the $t$-graphs on dihedral groups, let us first look at some useful lemmas.

\begin{lemma}\label{T3}
	Let $\mathcal{G} = (D_n,E)$ be the $t$-graph of  $D_n$. Then    
	\begin{equation}
		|E| = \begin{cases}
			4(n-t) + 2     & \text{If} \ \ t>1 \\
			3n - 2         & \text{If} \ \ t = 1.
		\end{cases}
	\end{equation}
\end{lemma}

\begin{proof}
	Let $x=a^ib^j, y=a^kb^l\in D_n$, then $0\leq i,k\leq 1$ and $0\leq j,l\leq n-1$. If $d_1(x,y) = |i-k|+|j-l|=t$, then for $|i-k|$ we have the following cases:
	\begin{enumerate}[(1)]
		\item If $i=k$ then $|j-l|=t$. Note  that there are $n-t$ ways to choose  $j, l\in \{0, \dotsc, n-1\}$ such that the absolute value of their difference   is $t$.
		
		\item  If $i\neq k$ then  $|j-l|=t-1$. In this case there are $n-t+1$ forms to choose $j, l\in \{0, \dotsc, n-1\}$ such that the absolute value of their difference   is  $t-1$.
		
	\end{enumerate}
	If $t>1$ then  there are $2(n - t) + 2(n - t + 1)$  ways of constructing an edge between two elements of $D_n$. Therefore we have that $|E| = 4(n - t) + 2$.
	
	If $t = 1$ then we the same argument we have that $|E| = 3n - 2$. 
\end{proof}

\begin{lemma}\label{T4}
	Let  $f: D_n\longrightarrow D_n$ defined as follows
	\begin{equation}
		f(a^ib^j) = \begin{cases}
			b^j      & \text{If \ $i=1$} \\
			ab^j     & \text{If \ $i = 0$}.
		\end{cases}
	\end{equation}
	Then $f$ is an isometry under the Minkowski metric 	\eqref{d1}. Further, if we restrict $f$ to $U\subset D_n$, we have that  $U$ and $f(U)$ are also isometric under the Minkowski metric.
\end{lemma}

\begin{proof}
	It is immediate that $f$ is an injective function and  $(f\circ f)(x) = x$, for all $x\in D_n$. That is, $f$ is bijective. To prove that $f$ is an isometry, let  $a^ib^j, a^kb^l\in D_n$. Then 
	\begin{enumerate}[(1)]
		\item   If $i, k=1$ then $d_1(f(a^ib^j), f(a^kb^l)) = d_1(b^j, b^l) = d_1(a^ib^j, a^kb^l)$.
		\item   If $i, k=0$ then similar to the previous case.
		\item  If $i = 0$ and $k=1$ then $d_1(f(a^ib^j), f(a^kb^l)) = d_1(ab^j, b^l) = d_1(a^ib^j, a^kb^l)$.
		\item  If $i = 1$ and $k=0$ then similar to the previous case.
	\end{enumerate}
	Therefore $f$ is an isometry on $D_n$. The other statement is clear 
\end{proof}

%
%

\begin{theorem}\label{T5}
	(\textbf{Characterisation of $t$-graphs on $D_n$})\\ 
	Let $\G=(D_n, E)$ the $t$-graph of $D_n$ with $n \geq 2$. We define $r:=\left\lceil \frac{n}{2} \right\rceil$. 
	\begin{enumerate}[(1)]
		\item If $t\leq r$ and $t$ is an even number then $k(\G) = 2$, and these connected components are isomorphic.
		
		\item  If $t\leq r$ and $t$ is an odd number then  $\G$ is an connected graph.
		
		\item If $t = r + s$, with $1\leq s \leq n-r$ then the number $K(\G)$ of connected components of $\G$ is given by  
		\begin{equation}
			k(\G) = \begin{cases}
				4(s-1) + 2      &  \text{If $n$ is even}  \\
				4s & \text{If $n$ is odd},
			\end{cases}
		\end{equation}
		where two of the connected components of $\G$ are isomorphic path graph.	
	\end{enumerate}
\end{theorem}

\begin{proof} 
	(1) It follows from Theorem \ref{T2}  that $k(\G) = 2$. The connected components of $\G$ are $\mathcal{C}_1 = (V_1, E_1)$ and $\C_2 = (V_2, E_2)$ like in the proof of Theorem \ref{T2}  (1). It is then sufficient to show that $\C_1\cong \C_2$.
	Using the function $f$ defined in Lemma \ref{T4} we have for $a^ib^j, a^kb^l\in V_1$ that 
	\begin{equation}
		\{a^ib^j, a^kb^l\}\in E_1 \Longleftrightarrow \{f(a^ib^j), f(a^kb^l)\}\in E_2,
	\end{equation}
	which leads to $\C_1\cong \C_2$.
	
	\noindent (2) This follows immediately from  Theorem  \ref{T2}. 
	
	\noindent (3) We differentiate two cases:
	\begin{enumerate}[(a)]
		\item Suppose $t$ is an even number. The condition $t>r$  implies that $\G$ has isolated points and then using Theorem \ref{T2} we have that $\G$ has at least two connected components. Let $\C_1 = (V_1, E_1)$ and $\C_2 = (V_2, E_2)$   the connected components constructed in the proof of  theorem \ref{T2} (1). 
		
		We prove first that $|V_1|=|V_2|$. In fact, we have that $|j-l|=t$ or $|j-l|=t-1$, which implies that  
		\begin{equation}
			j\in \{t-1, \dotsc, n-1\}\cup \{0, \dotsc, n-t\} =: A,
		\end{equation}
		since  $l\in \{0, \dotsc, n-1\}$. 
		
		It is clear that $\{t-1, \dotsc, n-1\}\cap \{0, \dotsc, n-t\}=\emptyset$, therefore 
		\begin{equation}
			|A|=2(n-t)+2.
		\end{equation}
		On the other hand, it follows immediately that $j\in A$ and $i+j$ is an even number if and only if $a^ib^j\in V_1$, and then 
		\begin{equation}
			|V_1| = |A| = 2(n-t)+2.
		\end{equation}
		Analogously  $|V_2|=|A|$  and we have $|V_1|=|V_2|$.
		
		To demonstrate that $\C_1 \cong \C_2$, we consider again the function $f$ defined in Lemma \ref{T4}. Note that $f(V_1)=V_2$, and since $f$ is an isometry we have the statement. 
		
		Finally,  using Lemma \ref{T3} we have that $|E| = 4(n-t)+2$, and   the isomorphy between $\C_1$ and $\C_2$ implies that $|E_1|=|E_2|$. Further, note that the minimum value for  $|E_1|$ and $|E_2|$ is $2(n-t)+1$. This prove that  $\C_1$ and $\C_2$ are the unique connected components of $\G$, which are not isolated points and these are actually isomorphic path graphs.
		
		The number of isolated points of $\G$ is $|D_n| - |V_1| - |V_2| = 2n - 4(n-t) - 4 = -2n+4t-4$, and consequently  $k(\G) = -2n+4t-2 = -2n+4r+4s-2$. That is,
		\begin{itemize}
			\item If $n$ is even, then $k(\G) =  -2n+4(\frac{n}{2})+4s-2 = 4(s-1) + 2$.
			\item  If $n$ is odd, then $k(\G) =  -2n+4(\frac{n+1}{2})+4s-2 = 4s$.
		\end{itemize}

		\item Suppose now $t$ is an odd number.  Similar as before, the graph $\G$ has isolated points, and the set
		\begin{equation}
			\{\{a^ib^j, a^kb^l\} \mid i+j\equiv 0 \bmod 2, k+l \equiv 1 \bmod 2 \wedge |i-k|+|j-l|=t \},
		\end{equation}
		is a subset of $E$. Let $V'$ be the set consisting of the non-isolated points of $\G$. Using the same argument as in (a) we get that 
		\begin{equation}
			|V'| = 2|A| = 4(n-t) + 4.
		\end{equation}
		By Lemma \ref{T3} we have that $|E|=4(n-t)+2$ , then comparing $|V'|$ and $|E|$, excluding isolated points, it follows that $\G$ cannot be connected. 
		
		Let $m$ be an even number such that 
		\begin{equation}
			\begin{cases}
				0\leq m\leq n-t-1      & \text{if $n-t-1$ is even, and} \\
				0\leq m\leq n-t & \text{if $n-t-1$ is odd},
			\end{cases}
		\end{equation}
		and consider the subgraph $\C_1 = (V_1, E_1)$ of $\G$ with the following edges:
		\begin{equation*}
			\{ab^{t+m-1}, b^m\},\ \{b^m,  b^{t+m}\},\ 	\{b^{t+m}, ab^{m+1}\},\  \{ab^{m+1},ab^{t+m+1}\}.
		\end{equation*}
		Then $\C_1$ is a connected component of $\G$, and furthermore 
		\begin{equation}
			|V_1| = 2(n-t)+2 \ \wedge \ |E_1| =2(n-t)+1,
		\end{equation}
		whence it is concluded that $\C_1$ is a  path graph.
		
		As before, using the function $f$ from Lemma \ref{T4} we have that there exist another connected component $\C_2 = (f(V_1), E_2)$, isomorphic to $\C_1$.  Thus, 
		\begin{equation}
			|E_1|+|E_2|=|E| \ \wedge \ |V_1|+|V_2|=|V'|.
		\end{equation}
		It means that  $\C_1$ and $\C_2$ are the unique connected components of $\G$, and analogously to the previous case, we have the same values for $k(\G)$. 
	\end{enumerate}
\end{proof}

The following corollary is a generalisation of Theorem \ref{E2}

\begin{corollary}\label{T6}
	Let $G$ be a 2-generator group in the form \eqref{2-gen} with  $n, m \geq 2$, and  $t$  an odd number.  Let further $r$ defined as in Theorem \ref{T2}. If $t\leq r$, then  $\G = (G,E)$ is a bipartite  graph.
\end{corollary}

\begin{proof}
	From Theorem  \ref{T5}  we have that $\G$ is connected. Now we define the sets $V_1$ and $V_2$ as follows:
	\begin{align}
		V_1 & := \{a^ib^j \mid  i+j \equiv 0 \bmod 2 \}\\
		V_2 & := \{a^ib^j\mid  i+j\equiv 1 \bmod 2 \}
	\end{align}
	It is immediate to verify that $V_1$ and $V_2$ form a bipartition of $G$, and $\G$ is a  bipartite  graph.
\end{proof}

An illustration of the previous Corollary is the following:    

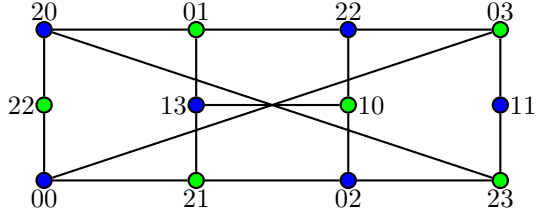
\begin{figure}[H]
	\begin{center}
		\begin{tikzpicture}[mynode/.style={font=\color{#1}\sffamily,circle,inner sep=2pt,minimum size=0.2cm,draw,thick}] 
			\node[mynode=yellow,fill=blue] (1) at (1,0) {};
			\node[mynode=yellow,fill=green] (2) at (3,0) {};
			\node[mynode=black,fill=blue] (3) at (5,0) {};
			\node[mynode=green,fill=green] (4) at (7,0) {}; 
			\node[mynode=yellow,fill=green] (1a) at (1,1) {};
			\node[mynode=yellow,fill=blue] (2a) at (3,1) {};
			\node[mynode=black,fill=green] (3a) at (5,1) {};
			\node[mynode=green,fill=blue] (4a) at (7,1) {}; 
			\node[mynode=yellow,fill=blue] (1b) at (1,2) {};
			\node[mynode=yellow,fill=green] (2b) at (3,2) {};
			\node[mynode=black,fill=blue] (3b) at (5,2) {};
			\node[mynode=green,fill=green] (4b) at (7,2) {}; 
			\draw[thick] (1) --(2) --(3) --(4)    (3a) -- (2a) (1b) -- (2b) -- (3b) -- (4b) -- (1) ;
			\draw[thick] (1) --(1a) -- (1b) -- (4) (2) --(2a) -- (2b)   (3) -- (3a) -- (3b)  (4) -- (4a) -- (4b)  ;
			
			\node [below] at (1,0) {$00$};
			\node [below] at (3,0) {$21$};
			\node [below] at (5,0) {$02$};
			\node [below] at (7,0) {$23$};
			\node [right] at (7,1) {$11$};
			\node [above] at (7,2) {$03$};
			\node [above] at (5,2) {$22$};
			\node [above] at (3,2) {$01$};
			\node [above] at (1,2) {$20$};
			\node [left] at (1,1) {$22$};
			\node [right] at (5,1) {$10$};
			\node [left] at (3,1) {$13$};
		\end{tikzpicture}
	\end{center}
	\caption{The 3-graph of $G= \Z_3\times \Z_4$.}
	\label{TF5}
\end{figure}

\begin{corollary}\label{T7}
	Let $n$ be an odd number, $n\geq 5$ and  $t = \frac{n+1}{2}$.  
	\begin{enumerate}[(1)]
		\item If $t$ is odd, then  $\G = (D_n,E)$  is a cycle of even length.		
		\item  If $t$ is even, then $\G = (D_n,E)$  is non connected, ant it has two isomorphic components, which are two cycles. Furthermore $\chi(\G) =3$.
	\end{enumerate}
\end{corollary}

\begin{proof}
	These statements follow direct from Theorem \ref{T5}. Note that  $t=r$.
	\begin{enumerate}[(1)]
		\item  From Lemma  \ref{T3} it follows  that
		\begin{equation}
			|E| = 4(n-(\tfrac{n+1}{2})) + 2 = 2n = |D_n|, 
		\end{equation}
		and then $\G$  is a cycle of even length.
		
		\item $\G$ has two isomorphic connected components, say  $\C_1 = (V_1, E_1)$ and  $\C_2 = (V_2, E_2)$. Lemma  \ref{T3} implies that \begin{equation}
			|E| = 4(n-(\tfrac{n+1}{2})) + 2 = 2n, 
		\end{equation}
		and follows that $|E_1| = |E_2| = n$, and so $\G$  is constituted by two isomorphic cycles. Finally, note that each component has an odd number of vertices, Then $\chi(\G) =3$.
	\end{enumerate}
\end{proof}

\begin{corollary}
	Let $n$ be an even number, $n\geq 2$ and  $t = \frac{n}{2} + 1$.  Then the  $t$-graph of $D_n$  consists of two isomorphic paths graphs.
\end{corollary}

\begin{proof}
	Using Theorem \ref{T5}, and since $n$ is an even number, we have that  $r = \frac{n}{2}$, and then $t = r + 1$, and $k(\G) = 2$. The rest is clear. 
\end{proof}

\begin{corollary}
	Let $n\geq 2$, and $r$ as in Theorem \ref{T5}. Then the $t$-graph of $D_n$ is 2-chromatic if  $t\leq r$ and $t$ is an odd number or  $t>r$.
\end{corollary}

\begin{proof}
	It follows immediately from  Theorem \ref{T5} and Corollary \ref{T6}.
\end{proof}

\begin{corollary}
	The $n$-graph of $D_n$  has $2(n-1)$  connected components and two of these are path graph wit 2 vertices.
\end{corollary}

\begin{proof}
	Let $\G = (D_n,E)$ be the $n$-graph of $D_n$. From Lemma  \ref{T3} it follows  that $	|E| = 2$. Note that 
	\begin{equation}
		\{a,  b^{n-1}\},\ 	\{ab^{n-1}, 1\}\in E.
	\end{equation}
	The other $2n-4$ elements of $D_n$ are isolated points, and the proof is complete.
\end{proof}

\section{Some conjectures}
Some open questions are presented below as conjectures.

\begin{conjeture}
	With respect to Theorem \ref{T2}, if $m$ is an even number and $t\leq r$ it follows that the two connected components of the $t$-graph $\G$ are isomorphic.
\end{conjeture}

\begin{conjeture}
	Is it possible to characterise the $t$-graphs on 2-generator groups, when $t> r$, and $r$ as in Theorem \ref{T2}?
\end{conjeture}

\begin{conjeture}
	Let $n\geq 2$, and $r$ as in Theorem \ref{T5}. Then the $t$-graph of $D_n$ is 3-chromatic if $t\leq r$ and $t$ is an even number. 
\end{conjeture}

\begin{conjeture}
	Is it possible to generalise a version of Theorem \ref{T2} for an $n$-generator group for $n$ an arbitrary natural number?
\end{conjeture}

\end{document}